\documentclass[11pt]{article}
\usepackage[T1]{fontenc}
\usepackage{lmodern,microtype,amsmath,amssymb,amsthm}
\usepackage[margin=1in]{geometry}
\usepackage[colorlinks=true,linkcolor=blue,citecolor=blue,urlcolor=blue]{hyperref}
\hypersetup{pdftitle={Pointwise provable equality and the failure of composition},
  pdfauthor={Florian Lengyel}, pdfsubject={Version 38}}
\newtheorem{theorem}{Theorem}
\newtheorem{corollary}[theorem]{Corollary}
\newtheorem{proposition}[theorem]{Proposition}
\newcommand{\PA}{\mathrm{PA}}
\newcommand{\Con}{\operatorname{Con}}
\newcommand{\Proof}{\operatorname{Proof}}
\newcommand{\comp}{\operatorname{comp}}
\title{Pointwise provable equality and the failure of composition}
\author{Florian Lengyel\thanks{Email: \texttt{flengyel@gradcenter.cuny.edu}}}
\date{November 14, 2026}
\begin{document}
\maketitle

\begin{abstract}
In their studies of pathologies in recursion categories, Montagna (1989) and Di Paola--Montagna (1991) introduce the algebraic
systems $S'$ and $S'_T$, respectively, and claim that they are categories. We show that the proposed composition is not independent 
of the choice of representatives. For every consistent recursively enumerable extension $T$ of Peano arithmetic ($\mathrm{PA}$), 
we exhibit two unary programs whose partial functions are provably equal in $T$, separately at each standard input. Composing each after a 
program that searches for a $T$-proof of contradiction and returns its code yields programs that are not
equivalent in this sense. An alternative proof uses the productivity of the complement of the diagonal halting set. 
Montagna's $S'$ is the case $T=\mathrm{PA}$. More generally, for consistent $T\supseteq\mathrm{PA}$, pointwise provable 
equality is a composition congruence exactly when $T$ proves every true $\Pi^0_1$ sentence, in which case it 
is extensional equality. This completeness condition fails for every consistent recursively enumerable $T\supseteq\mathrm{PA}$ 
by G\"odel's second incompleteness theorem. For every extension $T\supseteq\mathrm{PA}$, the least composition congruence 
containing pointwise provable equality is extensional equality if $T$ is $\Sigma^0_1$-sound and the universal relation otherwise.
\end{abstract}

\noindent\textit{2020 Mathematics Subject Classification.}\\
Primary 03D75; Secondary 03F40, 08A30, 18B10.

\section{Introduction}\label{sec:introduction}

Montagna and Di Paola--Montagna propose algebraic systems $S'$ and
$S'_T$, asserting that $S'$ is a pointed $p$-category
\cite[pp.~106, 108]{Montagna} and $S'_T$ a $p$-recursion category
\cite[p.~646]{DiPaolaMontagna}.
For every consistent recursively enumerable extension $T$ of
Peano arithmetic ($\PA$), we show that the proposed composition is not
independent of the choice of representatives.
The proposed morphisms of $S'$ and $S'_T$ are equivalence classes
of program indices under $T$-provable pointwise extensional equality
of partial functions defined by those indices, with composition induced
by composition of programs.

Montagna's $S$ \cite[p.~105]{Montagna} and Di Paola--Montagna's $S_T$
\cite[p.~646]{DiPaolaMontagna}, which use provability of a single
universally quantified equality statement, are unaffected by this obstruction.

\subsection{Relation to previous work}

Bergstra--Heering study $\omega$-complete algebraic specifications,
distinguishing derivability of every closed instance of an equation
from derivability of the open equation
\cite[Sections~1.2--1.3, pp.~150--152]{BergstraHeering}.
The analogous distinction here is between separate proofs of equality
at each standard input and a single proof of the universally quantified
equality.

Program indices with a primitive recursive composition operation form
a computable algebra with one binary operation. Computable quotient
presentations describe an algebra by computable operations on
representatives and an equivalence relation compatible with those
operations; the equivalence relation need not be computable
\cite[p.~9]{GodziszewskiHamkins}.
Theorem~\ref{thm:characterization} determines when pointwise provable
equality is a congruence for this operation, and
Theorem~\ref{thm:generated-congruence} identifies the least congruence
containing that equality.

Robertson--Ru\v{s}kuc--Thomson prove that the monoid of unary partial
recursive functions under composition is finitely generated but not
finitely presented
\cite[Theorems~2.2--2.3, pp.~168--169]{RobertsonRuskucThomson}.
Theorem~\ref{thm:generated-congruence} identifies this monoid as the
quotient by the composition congruence generated by pointwise provable
equality when $T$ is $\Sigma^0_1$-sound; otherwise the quotient has one
element.

\section{Programs and pointwise provable equality}\label{sec:programs}

\subsection{The constructions of Montagna and Di Paola--Montagna}

Fix an acceptable Kleene numbering $(\varphi_e)_{e\in\omega}$ of the unary
partial recursive functions and an arithmetization satisfying the
$\PA$-provable properties specified in
Section~\ref{sec:arithmetization}.
Throughout, a program is an index in this
numbering. Thus $e$ denotes an index, $\varphi_e$ denotes the partial function
computed by $e$, and $e(x)$ abbreviates $\varphi_e(x)$.
For partial expressions, $\simeq$ means that both are undefined or
both are defined with the same value.
Acceptability is understood
in the sense of \cite[Definitions~II.5.1--II.5.2, p.~215]{Odifreddi}.

In the proposed construction, $T$ is a consistent extension of $\PA$.
We take the object class and its effective coding from
\cite[pp.~105--106]{Montagna} and
\cite[p.~646]{DiPaolaMontagna}: the objects are generated from
$\omega$ by iterated binary products and disjoint unions.

For objects $A,B$, the proposed morphisms $A\to B$ of $S'_T$ are
classes of program indices, called G\"odel numbers in the cited papers,
that compute partial recursive functions $A\rightharpoonup B$.
Two indices $e,d$ are identified,
written $e\approx_{T,A}d$, precisely when, for every $a\in A$,
\begin{equation*}
 T\vdash
 \bigl[(\varphi_e(\bar a)\downarrow
          \leftrightarrow\varphi_d(\bar a)\downarrow)
 \land
 (\varphi_e(\bar a)\downarrow
          \to\varphi_e(\bar a)=\varphi_d(\bar a))\bigr].
\end{equation*}
This is the pointwise definition in \cite[p.~646]{DiPaolaMontagna}.
Here $\bar a$ denotes the numeral for the code of $a$; the quantification
over $A$ is external to $T$.
The symbols $\downarrow$ and $\uparrow$ denote convergence and
divergence, respectively.
Let $E_{e,d}(x)$ denote the displayed conjunction with $\bar a$ replaced
by $x$; it expresses Kleene equality at $x$.
Montagna's $S'$ is the case $T=\PA$; the divergence equivalence in
his first conjunct gives the same relation
\cite[p.~106]{Montagna}.

The proposed identity on $A$ is the class of an index computing its identity
function. Fix a primitive recursive composition-index function $\comp$ such
that $\varphi_{\comp(e,d)}(x)\simeq\varphi_e(\varphi_d(x))$ for all indices
$e,d$ and all $x$. Its existence follows from the enumeration and
$s$-$m$-$n$ theorems \cite[Theorem~II.1.5 and Proposition~II.1.7, pp.~130--131]{Odifreddi}.
For program indices $f,g$, write $f\bullet g:=\comp(f,g)$.
Thus $\bullet:\omega\times\omega\to\omega$ is a fixed primitive
recursive operation on indices, representing composition of partial
functions: $\varphi_{f\bullet g}=\varphi_f\circ\varphi_g$.
Here equality means equality of partial functions, including their domains
of definition. The program with index $g$, the \emph{inner program},
runs first; if it halts with output $y$, the program with index $f$,
the \emph{outer program}, runs on $y$.

If $d$ and $e$ compute partial functions
$A\rightharpoonup B$ and $B\rightharpoonup D$, respectively, the intended
composition is $[e]\circ[d]=[\comp(e,d)]$,
which must be independent of the choice of representatives.
No published correction of the pointwise definitions is known to the author.

\subsection{Programs and arithmetization}\label{sec:arithmetization}

For $A=B=\omega$, every index of a unary partial recursive function is
allowed as a representative. Write $\approx_T$ for
$\approx_{T,\omega}$.
For recursion-theoretic terminology, including recursive enumerability,
see \cite{Rogers,Odifreddi}; for the arithmetical hierarchy, see
\cite[Definition~IV.1.6, p.~367]{Odifreddi}.
Truth refers to the standard model $\omega$.
A theory is $\Sigma^0_1$-sound if every $\Sigma^0_1$ sentence it proves
is true.

Alongside pointwise provable equality, we use uniform provable equality:
\begin{align}
 e\equiv_T d
 &\quad\Longleftrightarrow\quad T\vdash\forall x\,E_{e,d}(x),
 \label{eq:uniform}\\
 e\approx_T d
 &\quad\Longleftrightarrow\quad
 \text{for every }n\in\omega,\quad T\vdash E_{e,d}(\bar n).
 \label{eq:pointwise}
\end{align}
The proposed endomorphisms of $\omega$ form the quotient set
$Q_T=\omega/{\approx_T}$; let $q_T$ be its quotient map.
Composition is independent of the choice of representatives precisely
when there is a binary operation $\star$ on $Q_T$ satisfying
\begin{equation}\label{eq:quotient-composition}
 q_T(\comp(e,d))=q_T(e)\star q_T(d)
 \qquad\text{for all }e,d.
\end{equation}
A composition congruence is an equivalence relation preserved by
$\comp$ in both arguments. Thus a binary operation satisfying
\eqref{eq:quotient-composition} exists exactly when $\approx_T$ is a
composition congruence.
We shall show that no such binary operation exists for consistent
recursively enumerable $T$.

The proofs use the following arithmetization.
Following Odifreddi
\cite[Theorem~II.1.2, p.~129]{Odifreddi}, write $\mathcal{T}_1(e,x,s)$ for Kleene's
primitive recursive normal-form predicate for unary programs, and $U$
for the associated primitive recursive output function.
In the index convention of \cite[Definition~II.1.4, p.~130]{Odifreddi},
\[
 \varphi_e(x)\simeq U\bigl(\mu s\,\mathcal{T}_1(e,x,s)\bigr).
\]
This is the unary case of Kleene's normal-form theorem
\cite[\S63, Theorem~XIX(a), p.~330]{Kleene}.
The third argument $s$ codes a halting computation of the program with
index $e$ on input $x$; its output is $U(s)$.
The subscript $1$ indicates that the program has one input.
We reserve $T$ for the arithmetical theory.
This is the
Kleene normal-form convention used by Di Paola--Montagna
\cite[p.~645]{DiPaolaMontagna}.

Define the arithmetical $\Sigma^0_1$ graph formula by
\[
 C_e(x,y)\;:\Longleftrightarrow\;
 \exists s\,\bigl(\mathcal{T}_1(\bar e,x,s)\land U(s)=y\bigr).
\]
Here $\mathcal{T}_1(\bar e,x,s)$ and $U(s)=y$ abbreviate fixed arithmetical
formulas representing in $\PA$ the computation predicate and the
graph of the output function, respectively.
Here $y$ denotes the output.
Convergence $\varphi_e(x)\downarrow$ abbreviates
$\exists y\,C_e(x,y)$, and divergence $\varphi_e(x)\uparrow$ its
negation. Functionality gives the $\PA$-provable equivalence
\[
 E_{e,d}(x)\longleftrightarrow
 \forall y\,\bigl(C_e(x,y)\longleftrightarrow C_d(x,y)\bigr).
\]
By the formalized enumeration and iteration theorems in $\PA$, invoked in
\cite[p.~647]{DiPaolaMontagna}, the fixed $\comp$ may be taken so that its
graph formula satisfies
\begin{equation}\label{eq:composition-graph}
 \PA\vdash\forall x\,\forall z\,
 \left(C_{\comp(e,d)}(x,z)\longleftrightarrow
       \exists y\,\bigl(C_d(x,y)\land C_e(y,z)\bigr)\right)
\end{equation}
for all indices $e,d$. The graph formulas are provably functional in
$\PA$. More generally, if $R(x,y)$ is a $\Sigma^0_1$ formula whose
functionality is provable in $\PA$, there is an index $e$ such that
\begin{equation}\label{eq:graph-realization}
 \PA\vdash\forall x\,\forall y\,
 \bigl(C_e(x,y)\longleftrightarrow R(x,y)\bigr).
\end{equation}
In applications of Feferman's results, we pass to
$\PA$-provably equivalent RE-formulas in his syntactic sense
\cite[p.~267]{Feferman}.
Choose such a formula $\widehat R(z,x,y)$ with free variables
exactly $z,x,y$ and
\[
 \PA\vdash\forall z\,\forall x\,\forall y\,
 \bigl(\widehat R(z,x,y)\longleftrightarrow R(x,y)\bigr).
\]
This is obtained by adjoining dummy equalities for any missing
variables and taking an RE normal form that retains those
free variables.
Its functionality is therefore provable in Feferman's primitive
recursive extension $M$ of $\PA$.
By \cite[Theorem~2.13, p.~273]{Feferman}, there is an index $e$
whose graph formula is provably equivalent in $M$ to
$\widehat R(\bar e,x,y)$, hence to $R(x,y)$.
Eliminate the auxiliary
primitive recursive definitions by \cite[Theorem~2.15, p.~274]{Feferman}
to obtain~\eqref{eq:graph-realization}. Montagna uses this specialization
in \cite[p.~109, proof of~$(*)$]{Montagna}.
The graph equations below are obtained from~\eqref{eq:graph-realization}.

\section{An obstruction to composition}\label{sec:obstruction}

In this section, assume that $T\supseteq\PA$ is consistent and recursively
enumerable. Fix an index
$a_T$ for a partial recursive function whose range is the set of G\"odel
codes of the nonlogical axioms of $T$.
We use Feferman's first-order calculus with equality, with the
logical axiom system chosen as in \cite[p.~266]{Feferman} so that
modus ponens is the only inference rule.
A derivation is a finite nonempty sequence of formulas, each a
logical axiom, a nonlogical axiom of $T$, or obtained from earlier
formulas by modus ponens; its conclusion is the last formula.
At each occurrence of a nonlogical
axiom with G\"odel code $r$, a derivation code records an input $m$ to this
enumerator and a finite computation-history code $s$ satisfying
$\mathcal{T}_1(a_T,m,s)$ and $U(s)=r$. Here $m$ is the enumerator input, $s$ is a
computation-history code, and the pair $(m,s)$ witnesses that $r$ is
enumerated as an axiom. The relation $\Proof_T(p,q)$, meaning
that $p$ codes a $T$-derivation whose conclusion has G\"odel code $q$, is
therefore primitive recursive. Put
\[
 \operatorname{Pr}_T(q):\!\equiv\exists p\,\Proof_T(p,q).
\]
For this derivation calculus, $\operatorname{Pr}_T$ satisfies the
Hilbert--Bernays--L\"ob derivability conditions; see
\cite[pp.~115--116]{Lob} for the conditions and
\cite[pp.~269--271, especially Theorems~2.4(i) and~2.8]{Feferman}
for the relevant arithmetization results for RE axiom numerations.
Write
$\ulcorner\sigma\urcorner$ for the G\"odel code of a sentence $\sigma$. Put
\[
 P_T(p):=\Proof_T(p,\ulcorner0=1\urcorner),
 \qquad
 \Con(T):=\forall p\,\neg P_T(p).
\]
Thus $P_T(p)$ says that $p$ is a code of a $T$-derivation of $0=1$.
The RE axiom numeration determined by $a_T$ is
\[
 \alpha_T(r):\!\equiv
 \exists m\,\exists s\,
 \bigl(\mathcal{T}_1(\overline{a_T},m,s)\land U(s)=r\bigr).
\]
Collecting the finitely many axiom witnesses, or deleting them
from an augmented derivation, gives $\PA$-provably equivalent
provability predicates for our coding and for $\alpha_T$.
Formalized proof manipulation, using a fixed $T$-proof of
$0\ne1$, then gives
$\PA\vdash\Con(T)\leftrightarrow\Con_{\alpha_T}$,
where $\Con_{\alpha_T}$ is Feferman's consistency sentence
\cite[p.~269]{Feferman}.
Let $z$ be a program whose graph is provably empty in $\PA$, and
let $\mathrm{id}$ be an identity program with
\[
 \PA\vdash\forall x\,\forall y\,
 \bigl(C_{\mathrm{id}}(x,y)\longleftrightarrow x=y\bigr).
\]

\begin{theorem}\label{thm:obstruction}
For every consistent recursively enumerable $T\supseteq\PA$, there are
programs $f,g$ satisfying
\begin{equation}\label{eq:obstruction-properties}
 f\approx_T\mathrm{id},\qquad
 f\bullet g\equiv_T z,\qquad
 \mathrm{id}\bullet g\equiv_T g,\qquad
 g\not\approx_T z.
\end{equation}
Consequently, no binary operation on $Q_T$ satisfies
\eqref{eq:quotient-composition}.
\end{theorem}

\begin{proof}
Choose
\[
 f(n)=
 \begin{cases}
 n,&\neg P_T(n),\\
 \uparrow,&P_T(n),
 \end{cases}
 \qquad
 g(x)=\mu p\,P_T(p).
\]
Choose their indices so that $\PA$ proves
\begin{align}
 C_f(n,y)&\longleftrightarrow \neg P_T(n)\land y=n,
 \label{eq:guard-graph}\\
 C_g(x,p)&\longleftrightarrow
       P_T(p)\land\forall r<p\,\neg P_T(r).
 \label{eq:search-graph}
\end{align}
Fix $n\in\omega$. Consistency of $T$ implies that $P_T(n)$ is false.
Since $P_T$ is primitive recursive, $\PA$ verifies this closed instance:
$\PA\vdash\neg P_T(\bar n)$. Hence
$\PA\vdash f(\bar n)=\bar n$.
Thus $f\approx_T\mathrm{id}$.

By~\eqref{eq:composition-graph}, \eqref{eq:guard-graph}, and
\eqref{eq:search-graph}, convergence of
$f\bullet g$ would require an intermediate value $p$ satisfying both
$P_T(p)$ and $\neg P_T(p)$. Hence
\begin{equation}\label{eq:composite-empty}
 \PA\vdash\forall x\,(f\bullet g)(x)\uparrow,
\end{equation}
so $f\bullet g\equiv_Tz$. The identity law
$\mathrm{id}\bullet g\equiv_Tg$ is also provable in $\PA$.

By~\eqref{eq:search-graph},
\[
 \PA\vdash
 \bigl(g(0)\uparrow\longleftrightarrow\Con(T)\bigr).
\]
If $g\approx_Tz$, the instance at input $0$ would therefore
yield $T\vdash\Con(T)$, contrary to G\"odel's second incompleteness
theorem \cite[Theorem~2.7(ii), p.~271]{Feferman}.
This proves all four assertions in \eqref{eq:obstruction-properties}.

Finally, if $\star$ satisfied~\eqref{eq:quotient-composition}, then
\[
 \begin{aligned}
 q_T(g)
 &=q_T(\mathrm{id}\bullet g)
   =q_T(\mathrm{id})\star q_T(g)\\
 &=q_T(f)\star q_T(g)
   =q_T(f\bullet g)
   =q_T(z),
 \end{aligned}
\]
contradicting $g\not\approx_Tz$.
\end{proof}

Under the consistency assumption, $\varphi_f=\varphi_{\mathrm{id}}$
and $\varphi_g=\varphi_z$: $f$ computes the identity and $g$ is nowhere
defined. Nevertheless, $g\not\approx_Tz$. The obstruction concerns
composition on pointwise provability classes; composition of the underlying
partial functions remains well defined.

\begin{corollary}\label{cor:not-category}
For consistent recursively enumerable $T\supseteq\PA$, the algebraic
system $S'_T$ is not a category under the proposed composition.
The failure already occurs for indices representing partial maps
$\omega\rightharpoonup\omega$. In particular, it occurs for
Montagna's $S'$ with $T=\PA$.
\end{corollary}

The categorical associativity and identity axioms are not needed for
this contradiction.
Any composition congruence containing $\approx_T$ must identify $g$
with $z$, and therefore strictly enlarge $\approx_T$.
Theorem~\ref{thm:generated-congruence} determines the least such congruence.

The assertion that $S'_T$ is locally connected if and only if $T$
is $\Sigma^0_1$-sound \cite[Theorem~4.5, pp.~658--659]{DiPaolaMontagna}
presupposes that $S'_T$ is a category.

\subsection{A proof from productivity}\label{sec:productive}

Di Paola--Heller prove a creativity theorem for the diagonal halting
domain in dominical recursion categories
\cite[Theorem~8.5, pp.~622--623]{DiPaolaHeller}.

The existence assertion in Theorem~\ref{thm:obstruction} also follows
from the productivity of the complement of the diagonal halting set
\cite[Section~5.3, p.~136]{GallierQuaintance}.
Put
\[
 K=\{e:\varphi_e(e)\downarrow\},\qquad
 D_T=\{e:T\vdash\varphi_e(\bar e)\uparrow\},
 \qquad W_e=\operatorname{dom}(\varphi_e).
\]
The set $D_T$ is recursively enumerable. Consistency gives
$D_T\subseteq\overline K$, since every actual halting computation is
verifiable in $\PA$. Choose $d$ with $W_d=D_T$.
If $d\in D_T$, then $\varphi_d(d)\downarrow$, contradicting
$D_T\subseteq\overline K$. Hence
\[
 d\in\overline K\setminus D_T.
\]
This is the productive-set argument: whenever $W_e\subseteq\overline K$,
the index $e$ itself belongs to $\overline K\setminus W_e$.

Let $R(n):\!\equiv\neg\mathcal{T}_1(\bar d,\bar d,n)$, and choose programs
\[
 f(n)=
 \begin{cases}
 n,&R(n),\\
 \uparrow,&\neg R(n),
 \end{cases}
 \qquad g(x)=\mu n\,\neg R(n),
\]
with $\PA$-provable graph equations
\[
 C_f(n,y)\longleftrightarrow R(n)\land y=n,
 \qquad
 C_g(x,n)\longleftrightarrow
 \neg R(n)\land\forall k<n\,R(k).
\]
Every standard instance $R(\bar n)$ is $\PA$-provable, so
$f\approx_T\mathrm{id}$. The graph equations give
$f\bullet g\equiv_Tz$ and $\mathrm{id}\bullet g\equiv_Tg$.
By the definition of $C_d$ and the $\PA$-provable totality of $U$,
\[
 \PA\vdash
 \bigl(\varphi_d(\bar d)\uparrow\longleftrightarrow\forall n\,R(n)\bigr).
\]
Together with the graph equation for $g$ and the definition of $D_T$,
this gives
\[
 g\approx_Tz
 \quad\Longleftrightarrow\quad T\vdash\forall n\,R(n)
 \quad\Longleftrightarrow\quad d\in D_T,
\]
which is false. Thus all four assertions
of~\eqref{eq:obstruction-properties} hold.
The proof using second incompleteness selects the particular unprovable
sentence $\Con(T)$; productivity supplies another true unprovable
divergence sentence.

\section{When pointwise provable equality is a congruence}\label{sec:characterization}

For any $T\supseteq\PA$, left composition with a fixed program $f$
is compatible with $\approx_T$: replacing the inner program gives
$g\approx_T h\Rightarrow f\bullet g\approx_T f\bullet h$.
Indeed, for each $n\in\omega$, substitute the theorem
$T\vdash\forall y\,(C_g(\bar n,y)\leftrightarrow C_h(\bar n,y))$
into~\eqref{eq:composition-graph}. Right composition with a fixed program $g$
instead replaces the outer program: $f\approx_T h$ would have to imply
$f\bullet g\approx_T h\bullet g$. Theorem~\ref{thm:obstruction} gives
$f\approx_T\mathrm{id}$ but
$f\bullet g\not\approx_T\mathrm{id}\bullet g$.
The family $T\vdash E_{f,h}(\bar m)$ $(m\in\omega)$ need not yield a
single $T$-derivation of $\forall y\,E_{f,h}(y)$. Such a derivation would justify
replacing $C_f(y,z)$ by $C_h(y,z)$ beneath the existential quantifier over
the intermediate output in~\eqref{eq:composition-graph}.

\begin{theorem}\label{thm:characterization}
For any consistent $T\supseteq\PA$, not necessarily recursively
enumerable, the following are equivalent.
\begin{enumerate}
\item\label{it:right}
Pointwise provable equality is preserved by right composition:
$f\approx_T h$ implies $f\bullet g\approx_T h\bullet g$ for every
program index $g$.
\item\label{it:congruence}
The relation $\approx_T$ is a composition congruence.
\item\label{it:pi}
The theory $T$ proves every true $\Pi^0_1$ sentence.
\item\label{it:extensional}
For all $e,d$, $e\approx_T d$ if and only if $\varphi_e$ and
$\varphi_d$ are extensionally equal as partial functions on $\omega$.
\end{enumerate}
\end{theorem}

\begin{proof}
We first prove \ref{it:right}$\Rightarrow$\ref{it:pi}.
Let $\theta$ be a true $\Pi^0_1$ sentence.  Choose a primitive recursive
predicate $R$ with a decision procedure whose correctness is provable
in $\PA$, such that
\[
 \PA\vdash\theta\longleftrightarrow\forall m\,R(m).
\]
Define programs
\[
 f(n)=
 \begin{cases}
 n,&R(n),\\
 \uparrow,&\neg R(n),
 \end{cases}
 \qquad
 g(x)=\mu m\,\neg R(m).
\]
Choose their indices so that $\PA$ proves
\begin{align}
 C_f(n,y)&\longleftrightarrow R(n)\land y=n,
 \label{eq:f-graph}\\
 C_g(x,m)&\longleftrightarrow
        \neg R(m)\land\forall k<m\,R(k).
 \label{eq:g-graph}
\end{align}
For each $n\in\omega$, the $\PA$-verified decision procedure for $R$ gives
$\PA\vdash R(\bar n)$. Hence
$\PA\vdash f(\bar n)=\bar n$, and $f\approx_T\mathrm{id}$.

Equations~\eqref{eq:composition-graph}, \eqref{eq:f-graph}, and
\eqref{eq:g-graph} show that
$\PA$ proves that $f\bullet g$ is everywhere undefined:
any intermediate output $m$ would have to satisfy both $\neg R(m)$
and $R(m)$. Thus $f\bullet g\equiv_Tz$.
By~\eqref{eq:g-graph},
\[
 \PA\vdash\forall x\,
 \bigl(g(x)\uparrow\longleftrightarrow\theta\bigr).
\]
Consequently,
\[
 g\approx_Tz\quad\Longleftrightarrow\quad T\vdash\theta.
\]
The forward implication uses the instance at input $0$;
the reverse implication follows from the displayed equivalence.
Assumption~\ref{it:right}, applied to
$f\approx_T\mathrm{id}$, now gives $f\bullet g\approx_Tg$ by the
$\PA$-provable identity law. Hence $g\approx_Tz$ and $T\vdash\theta$.

Next assume~\ref{it:pi}, and fix indices $e,d$ such that $\varphi_e$ and
$\varphi_d$ are extensionally equal. Fix $n\in\omega$. If
$\varphi_e(n)=\varphi_d(n)=y$, computation-history codes yield $\PA$-proofs
of $C_e(\bar n,\bar y)$ and $C_d(\bar n,\bar y)$; $\PA$-provable functionality
then yields $\PA\vdash E_{e,d}(\bar n)$. If both functions are undefined at
$n$, the two divergence assertions are true $\Pi^0_1$ sentences, so $T$
proves them and hence proves $E_{e,d}(\bar n)$. Thus extensional equality
implies $\approx_T$.

Conversely, suppose that $\varphi_e$ and $\varphi_d$ are extensionally
unequal. Choose $n\in\omega$ at which they differ. If both are defined there
with distinct outputs, their computation-history codes and $\PA$-provable
functionality yield a $\PA$-proof of $\neg E_{e,d}(\bar n)$. If exactly one
is defined there, its computation-history code yields a $\PA$-proof of the
corresponding graph instance, while the other program's divergence at $n$ is
a true $\Pi^0_1$ sentence and hence is provable in $T$. Thus
$T\vdash\neg E_{e,d}(\bar n)$ in either case. Consistency excludes
$T\vdash E_{e,d}(\bar n)$.
This proves~\ref{it:extensional}.

Finally, extensional equality is a composition congruence, so
\ref{it:extensional}$\Rightarrow$\ref{it:congruence}, and
\ref{it:congruence}$\Rightarrow$\ref{it:right} is immediate.
\end{proof}

For consistent recursively enumerable $T\supseteq\PA$,
$\Con(T)$ is a true $\Pi^0_1$ sentence unprovable in $T$ by
G\"odel's second incompleteness theorem. Thus condition~\ref{it:pi}
fails, and $\approx_T$ is not a composition congruence.
Taking $\theta=\Con(T)$ in the first implication recovers the
programs of Theorem~\ref{thm:obstruction}.
The theory obtained by adjoining all true $\Pi^0_1$ sentences to
$\PA$ satisfies condition~\ref{it:pi}; it is consistent and is not
recursively enumerable.

\section{The generated composition congruence}\label{sec:generated}

For an extension $T\supseteq\PA$, let $\sim_T$ be the least
composition congruence on program indices containing $\approx_T$;
it is the intersection of all such congruences.

\begin{theorem}[Generated composition congruence]\label{thm:generated-congruence}
Let $T\supseteq\PA$, not necessarily consistent or recursively enumerable.
Then
\[
 {\sim_T}=
 \begin{cases}
 \{(e,d)\in\omega^2:\varphi_e=\varphi_d\},
       &\text{if $T$ is $\Sigma^0_1$-sound},\\
 \omega\times\omega,&\text{otherwise}.
 \end{cases}
\]
\end{theorem}

\begin{proof}
First we prove that extensional equality is contained in $\sim_T$.
Fix $e,d$ with $\varphi_e=\varphi_d$.
Choose a primitive recursive bijective pairing $\langle s,x\rangle$
with primitive recursive projections whose inverse identities are
provable in $\PA$. Put
\[
 D_e(x,s):\!\equiv
 \mathcal{T}_1(\bar e,x,s)\land\forall r<s\,\neg\mathcal{T}_1(\bar e,x,r).
\]
Choose program indices $H_e,B_e,A_{e,d}$ so that $\PA$ proves the
universal closures of
\begin{align*}
 C_{H_e}(x,n)&\longleftrightarrow
       \exists s\,\bigl(D_e(x,s)\land n=\langle s,x\rangle\bigr),\\
 C_{B_e}(\langle s,x\rangle,y)&\longleftrightarrow
       \mathcal{T}_1(\bar e,x,s)\land U(s)=y,\\
 C_{A_{e,d}}(\langle s,x\rangle,y)&\longleftrightarrow
       \mathcal{T}_1(\bar e,x,s)\land U(s)=y\land C_d(x,y).
\end{align*}
These are $\Sigma^0_1$ relations whose functionality is provable in
$\PA$, so~\eqref{eq:graph-realization} supplies the required indices.
The program $H_e$ returns the pair consisting of the least
computation-history code for $e$ on its input and the input itself.

For each standard pair $\langle s,x\rangle$, $\PA$ proves the true
instance of either $\mathcal{T}_1(e,x,s)$ or its negation. If
$\mathcal{T}_1(e,x,s)$ is false, the graph equations prove that both $A_{e,d}$ and $B_e$
are undefined there. If it is true, then
$\varphi_d(x)=\varphi_e(x)=U(s)$. A computation-history code for
$d$ on input $x$ therefore gives a $\PA$-proof of
$C_d(\bar x,\overline{U(s)})$. The graph equations and functionality
give $\PA$-provable equality of the two programs at that pair.
Thus $A_{e,d}\approx_{\PA}B_e$.

The least-number principle and functionality of $C_e$ give
\[
 \PA\vdash\forall x\,\forall y\,
 \left(C_e(x,y)\longleftrightarrow
       \exists s\,\bigl(D_e(x,s)\land U(s)=y\bigr)\right).
\]
Indeed, if $C_e(x,y)$ holds, the least $s$ satisfying
$\mathcal{T}_1(\bar e,x,s)$ also satisfies $C_e(x,U(s))$, so functionality
forces $U(s)=y$. Equation~\eqref{eq:composition-graph} now gives
$B_e\bullet H_e\equiv_{\PA}e$ and
\[
 \PA\vdash\forall x\,\forall y\,
 \left(C_{A_{e,d}\bullet H_e}(x,y)\longleftrightarrow
             C_e(x,y)\land C_d(x,y)\right).
\]
Since $\equiv_{\PA}\,\subseteq\,\approx_T\,\subseteq\,\sim_T$,
composition compatibility gives
$e\sim_T A_{e,d}\bullet H_e$.
Repeating the construction with $e,d$ interchanged gives
$d\sim_T A_{d,e}\bullet H_d$. The two composite graph formulas are
provably equivalent in $\PA$, by symmetry of conjunction. Hence
$e\sim_Td$.

Suppose next that $T$ is $\Sigma^0_1$-sound. If $e\approx_Td$ and
$\varphi_e(n)=m$, a computation-history code gives
$\PA\vdash C_e(\bar n,\bar m)$. Together with
$T\vdash E_{e,d}(\bar n)$ this yields
$T\vdash C_d(\bar n,\bar m)$. The latter sentence is $\Sigma^0_1$,
so it is true and $\varphi_d(n)=m$. Interchanging $e,d$ shows that
$\approx_T$ is contained in extensional equality. Extensional equality
is a composition congruence, so minimality of $\sim_T$, together with
the first part, proves the first case.

Finally, suppose that $T$ proves a false $\Sigma^0_1$ sentence
$\sigma$. Choose a program $\delta$ such that
\[
 \PA\vdash\forall x\,\forall y\,
 \bigl(C_\delta(x,y)\longleftrightarrow\sigma\land y=x\bigr).
\]
This graph is $\Sigma^0_1$ and provably functional in $\PA$.
Since $T\vdash\sigma$, we have $\delta\equiv_T\mathrm{id}$.
Since $\sigma$ is false, $\varphi_\delta=\varphi_z$, so the first
part gives $\delta\sim_Tz$. Thus $\mathrm{id}\sim_Tz$, and for
every program index $e$,
\[
 e\equiv_{\PA}e\bullet\mathrm{id}
   \sim_Te\bullet z\equiv_{\PA}z.
\]
Therefore $\sim_T$ is the universal relation.
\end{proof}

The quotient $\omega/{\sim_T}$ is therefore the monoid of unary partial
recursive functions when $T$ is $\Sigma^0_1$-sound, and the one-element
monoid otherwise. In particular, closing $\approx_{\PA}$ under composition 
gives exactly extensional equality.

The universal relation in the second case is the generated congruence 
$\sim_T$. If $T$ is consistent but not $\Sigma^0_1$-sound, it proves a
false $\Sigma^0_1$ sentence $\sigma$. Its negation is a true $\Pi^0_1$
sentence that $T$ cannot prove, by consistency.
Theorem~\ref{thm:characterization} therefore shows that $\approx_T$ is
not a composition congruence, whereas
Theorem~\ref{thm:generated-congruence} makes $\sim_T$ universal.
Moreover, $\approx_T$ is not universal: consistency prevents the
constant-zero and constant-one programs from being equivalent, since
$\PA$ proves that their values differ at input $0$.
If $T$ is inconsistent, both relations are universal; this case is
excluded from Theorem~\ref{thm:characterization}.

\section{Weak totality and ranges}\label{sec:applications}

For consistent recursively enumerable $T\supseteq\PA$, weak totality
and the range assignment \cite[Definition~1.1, p.~106, and p.~111]{Montagna}
are not invariant under $\approx_T$.

\subsection{Weak totality}

For a program index $e$, define the predicate
\begin{equation}\label{eq:weak}
 W_T(e)\quad\Longleftrightarrow\quad
 \text{for every program index }d,\quad
 \comp(e,d)\approx_Tz\ \Longrightarrow\ d\approx_Tz.
\end{equation}
This is the index formulation of weak totality in Di Paola--Montagna
\cite[Definition~1.5, p.~645]{DiPaolaMontagna}. Montagna calls a morphism
satisfying this condition DPH total \cite[Definition~1.1, p.~106]{Montagna}.

Let $d_e$ be a program index for the partial identity on the domain of
$\varphi_e$, with $\PA$-provable graph equation
\[
 C_{d_e}(x,y)\longleftrightarrow x=y\land\exists v\,C_e(x,v).
\]
The condition $d_e\approx_T\mathrm{id}$ is the index condition
corresponding to R totality in Montagna's terminology
\cite[Definitions~1.6--1.7, p.~107]{Montagna} and to totality in the
terminology of Di Paola--Montagna
\cite[Definition~1.4, p.~645]{DiPaolaMontagna}.

If $e\approx_T h$, then $d_e\approx_T d_h$: at each standard input,
provable Kleene equality gives provable equivalence of convergence,
and the displayed graph equations give equality of the domain partial
identities. Hence the R-totality condition is invariant under $\approx_T$.

\begin{proposition}\label{prop:weak}
For consistent recursively enumerable $T\supseteq\PA$, the predicate
$W_T$ is not invariant under $\approx_T$.
\end{proposition}

\begin{proof}
The $\PA$-provable identity law gives $W_T(\mathrm{id})$.
For $f,g$ in Theorem~\ref{thm:obstruction},
$f\bullet g\approx_Tz$ but $g\not\approx_Tz$, so $W_T(f)$ fails,
although $f\approx_T\mathrm{id}$.
\end{proof}

For $T=\PA$, the graph equation for $f$ has the form used for the first
program in Montagna's Example~2 \cite[p.~112]{Montagna}, with the chosen
predicate $P_{\PA}(x)$ in place of Montagna's
$\mathrm{Prf}_{\PA}(x,\ulcorner0=1\urcorner)$. Montagna defines
$\mathrm{Prf}_{\PA}$ relative to a chosen binumeration of the axioms of
$\PA$ \cite[note~2, p.~115]{Montagna}. Montagna states that the class of
his first program is R total and DPH total in $S'$. The corresponding
program $f$ above satisfies the R-totality condition
$d_f\approx_{\PA}\mathrm{id}$ but fails $W_{\PA}$.

The second program displayed in Montagna's example is the partial identity
with domain
\[
 \{p:\mathrm{Prf}_{\PA}(p,\ulcorner0=1\urcorner)\},
\]
the codes, in Montagna's proof coding, of $\PA$-derivations of $0=1$. It is
pointwise provably equal to $z$. By contrast, the unbounded-search program
$g$ defined above satisfies $f\bullet g\approx_{\PA}z$ and
$g\not\approx_{\PA}z$.

An arithmetical sentence $\sigma$ is $\Pi^0_1$-conservative over
$\PA$ if every $\Pi^0_1$ sentence provable in $\PA+\sigma$ is
already provable in $\PA$.
In \cite[Theorem~2.3, p.~111]{Montagna}, condition~(2) says that
the convergence sentence $f(\bar n)\downarrow$ is
$\Pi^0_1$-conservative over $\PA$ for every $n\in\omega$.
Each such sentence is already $\PA$-provable, so this condition holds.
Condition~(1), read as~\eqref{eq:weak}, fails.
Thus the implication (2)$\Rightarrow$(1) fails when interpreted in
terms of program indices.

\subsection{The range assignment}

For each program index $e$, choose a program index $\rho_e$ for the partial
identity whose domain is the set-theoretic range of $\varphi_e$. With $r$ as
input and $w$ as output, its $\PA$-provable graph equation is
\begin{equation}\label{eq:range}
 C_{\rho_e}(r,w)\longleftrightarrow
 r=w\land\exists x\,C_e(x,r).
\end{equation}
This is the range construction defined in
\cite[p.~111]{Montagna}; the range assertion for $S'_T$ is repeated in
\cite[p.~646]{DiPaolaMontagna}.

\begin{proposition}\label{prop:range}
For consistent recursively enumerable $T\supseteq\PA$, there is a program
$u$ with $u\approx_Tz$ but $\rho_u\not\approx_T\rho_z$.
Thus~\eqref{eq:range} does not define an operation on $Q_T$.
\end{proposition}

\begin{proof}
Let
\[
 u(p)=
 \begin{cases}
 0,&P_T(p),\\
 \uparrow,&\neg P_T(p)
 \end{cases},
\]
with $\PA$-provable graph equation
\[
 \forall p\,\forall y\,
 \bigl(C_u(p,y)\longleftrightarrow P_T(p)\land y=0\bigr).
\]
For every $p\in\omega$, $\PA$ proves $\neg P_T(\bar p)$ and
$u(\bar p)\uparrow$. Hence $u\approx_Tz$.
But~\eqref{eq:range} yields
\[
 \PA\vdash\forall r\,\forall w\,
 \left(C_{\rho_u}(r,w)\longleftrightarrow
       r=w\land r=0\land\exists p\,P_T(p)\right).
\]
In particular,
$\PA\vdash C_{\rho_u}(0,0)\leftrightarrow\exists p\,P_T(p)$.
The program $\rho_z$ is $\PA$-provably everywhere undefined.
If $\rho_u\approx_T\rho_z$, the instance at input $0$ gives
$T\vdash\neg\exists p\,P_T(p)$, again contradicting the second
incompleteness theorem.
\end{proof}

\paragraph{Use of AI tools.}
The author used ChatGPT\footnote{Some draft reviews used the
HostileRefereeGPT instructions:
\url{https://github.com/flengyel/HostileRefereeGPT}.}
and Claude to explore mathematical arguments and related literature
and to revise drafts. The author takes responsibility for the
mathematical arguments and references.

\end{document}